\documentclass[letterpaper]{amsart}

\usepackage[usenames, dvipsnames]{color}

\usepackage[english]{babel}
\usepackage{amsmath}
\usepackage{amssymb}
\usepackage{amsfonts}
\usepackage{mathrsfs}
\usepackage{graphicx}
\usepackage[colorinlistoftodos]{todonotes}
\usepackage{marvosym}
\usepackage{hyperref}
\hypersetup{
    colorlinks,
    linkcolor={red},
    citecolor={blue},
    urlcolor={red}
}

\usepackage{tikz}
\usetikzlibrary{arrows}

 \newtheorem{thm}{Theorem}[section]
 \newtheorem{cor}[thm]{Corollary}
 \newtheorem{lem}[thm]{Lemma}
 \newtheorem{prop}[thm]{Proposition}
 \newtheorem{quest}[thm]{Question}
 \theoremstyle{definition}
 \newtheorem{defn}[thm]{Definition}
 \theoremstyle{remark}
 \newtheorem{rem}[thm]{Remark}
 \newtheorem{ex}[thm]{Example}
 \numberwithin{equation}{section}

\usepackage{ulem}
\usepackage{stackengine}
\stackMath

\title[Mergelyan type theorems in several complex variables]{A function algebra providing new Mergelyan type theorems in several complex variables}

\author[J. Falc\'o]{Javier Falc\'o}
\address[Javier Falc\'o]{Departamento de An\'alisis Matem\'atico, Universidad de Valencia, Doctor Moliner 50, 46100 Burjasot (Valencia),
Spain.}
\email{Francisco.J.Falco@uv.es}

\author[P. M. Gauthier]{Paul M. Gauthier }
\address[Paul M. Gauthier]{D\'epartement de math\'ematiques et de statistique, Universit\'e de Montr\'eal, Montr\'eal, Qu\'ebec, Canada H3C3J7.}
\email{gauthier@dms.umontreal.ca}

\author[M. Manolaki]{Myrto Manolaki}
\address[Myrto Manolaki]{School of Mathematics and Statistics, University College Dublin, Belfield, Dublin 4, Ireland.} \email{arhimidis8@yahoo.gr}

\author[V. Nestoridis]{Vassili Nestoridis}
\address[Vassili Nestoridis]{Department of Mathematics, University of Athens, 157 84 Panepistemiopolis,  Athens, Greece.}
\email{vnestor@math.uoa.gr}

\date{\today}
\thanks{First author was supported by MINECO and FEDER Project MTM2017-83262-C2-1-P. Second author was supported by NSERC (Canada) grant RGPIN-2016-04107.}

\begin{document}
\baselineskip=.55cm

\begin{abstract}
	For compact sets $K\subset \mathbb C^{d}$, we introduce a subalgebra $A_{D}(K)$ of $A(K)$, which allows us to obtain Mergelyan type theorems for products of planar compact sets as well as for graphs of functions.
\end{abstract}

\maketitle

\section{Introduction}
In one complex variable, approximation theory is well developed \cite{Gaier,MR0864364}. In particular, we have the celebrated theorems of Runge and Mergelyan. The present paper deals with approximation in several variables, where the situation is far from being understood \cite{FFW,MR2276419}. 

Runge's Theorem tells us that  every function holomorphic on a neighbourhood of a given planar compact set  $K$ can be approximated uniformly on $K$ by rational functions of one complex variable. The direct analogue of Runge's Theorem fails in several complex variables. However, in Section \ref{approxspecialtype}, we provide a Runge type theorem for approximation on products of planar compact sets by rational functions of special type.

Mergelyan type approximation is much stronger than Runge type approximation in the sense that the functions to be approximated are no longer assumed to be holomorphic on neighbourhoods of $K$. This relates to the ubiquitous algebra $A(K)$ \cite{MR0410387, MR2305474}, which in both one and several complex variables consists of all functions $f:K\to \mathbb C$ continuous on the compact set  $K$ and holomorphic on the interior of $K$, which is denoted by $K^{\circ}$. In particular, if $K^{\circ}=\emptyset$, then $A(K)=C(K)$. In sections \ref{Mergelyanresult} and \ref{Applications} we present stronger approximation results in the spirit of Mergelyan.

A direct consequence of \cite[Corollary 6.4]{MR0249639} gives that every function in $A(K_{1}\times K_{2})$ can be approximated by functions holomorphic on open sets containing $K_{1}\times K_{2}$, provided that $K_{1}$ and $K_{2}$ are planar compact sets and rational functions of one complex variable with poles off $K_{i}$ (equivalently, functions holomorphic on open sets containing $K_i$) are dense in $A(K_{i})$ for $i=1,2$.  This result has already been used in \cite{FalcoNestoridis2017,FalcoNestoridis2018,Gauthier2014} for products of  regular closed planar compact sets  $K_{i}$, that is $\overline{K_{i}^{\circ}}=K_{i}$. In Section \ref{Ac(K)}, we give an example showing that  \cite[Corollary 6.4]{MR0249639} does not always hold. Motivated by this, for compact sets  $K$ in $ \mathbb C^{d}$, we shall introduce an algebra $A_{D}(K)$ for which the previous example is no longer a counterexample. Recalling that a mapping from a planar domain to $\mathbb C^d$ is holomorphic if each coordinate is a complex-valued holomorphic function, we define the 
algebra $A_D(K)$ as the set of all functions $f:K\to \mathbb C$ that are continuous and such that,  for every open disc $D\subset \mathbb C$ and every injective holomorphic mapping $\Phi: D\to K\subset \mathbb C^{d}$, the composition $f\circ \Phi:D\to \mathbb C$ is holomorphic. Thus, $f$ should be holomorphic on every complex manifold contained in $K$ that is immersed in $\mathbb C^{d}$.  This algebra is a uniform Banach algebra endowed with the supremum norm. Since the new algebra $A_{D}(K)$ is contained in $A(K)$, we hope that it will become easier to obtain more Mergelyan type theorems in several complex variables.

In Section \ref{Mergelyanresult} we prove two such theorems. Firstly, we prove that \cite[Corollary 6.4]{MR0249639} holds in full generality if we replace $A(K)$ by  $A_{D}(K)$.  Furthermore, if $K_{i}$ are regular closed planar compact sets for all $i=1,\ldots, d$, then the algebras $A(K_{1}\times\cdots\times K_{d})$ and $A_{D}(K_{1}\times\cdots\times K_{d})$  coincide and consequently \cite{FalcoNestoridis2017,FalcoNestoridis2018,Gauthier2014} are justified. Secondly, we prove that if $K$ are graphs of certain functions, Mergelyan's Theorem holds for $A_{D}(K)$ while it fails for $A(K)$.

In Section \ref{Applications} we consider $K$  to be arbitrary products of planar compact sets, even infinitely many, and also disjoint unions of such products. We provide results on approximation of functions in $A_{D}(K)$ by rational functions of special type.

Finally, in Section \ref{appendix} we provide some examples which clarify that the inclusions among all the algebras we consider are in general strict.

In future papers, we shall explore more general Mergelyan type approximation by imposing more natural restrictions on the class of plausibly approximable functions. Denote the polynomially convex hull of  $K$ by $\widehat K$ and the rationally convex hull of $K$ by $K^{\wedge_{r}}$.  We shall (in these future papers) also define the $\mathcal O(K)$-$hull$ appropriately.  It is natural to restrict the class of  functions  plausibly approximable on $K$ by polynomials to the class of functions on $K$ having an extension in $A_D(\widehat K).$  For approximation by rational functions we replace the polynomial hull of $K$ by its rational hull and for approximation by holomorphic functions on open sets containing $K$ by  the $\mathcal O(K)$-hull of $K$. 

\textbf{Acknowledgement.} We wish to thank T. Hatziafratis and  A. Katavolos for helpful communications. We would also like to thank  E. Zeron for suggesting Example \ref{zeronex}. This work was completed during our stay in 2018 at the Mathematisches Forschungsinstitut Oberwolfach in the framework of the ``Research in Pairs'' program. We thank the institute for the generous hospitality which gave us the opportunity to realize this project.
\section{Approximation by rational functions of special type}
\label{approxspecialtype} 
This paper deals with several classes of functions on compact sets $K$  in $\mathbb C^{d}$. We denote by $P(K)$ the closure of the set of polynomials in $C(K)$.  We  also denote by $r(K)$ the set of all  rational functions with no singular points on $K$ and by  $R(K)$ the closure of $r(K)$ in $C(K)$. Finally let $\mathcal{O}(K)$ be the set of  functions $g$ holomorphic on an open set $V_{g}$ such that $K\subset V_{g}\subset \mathbb C^{d}$ and let $\overline{\mathcal O}(K)$ be the closure of $\mathcal{O}(K)$ in $C(K)$.

We have the following inclusions:
$$		
P(K) \subset R(K) \subset \overline{\mathcal O}(K) \subset A_{D}(K) \subset A(K) \subset C(K).
$$
The object of complex approximation is to determine under which conditions we have equalities among some of these inclusions. The introduction of the class $A_{D}(K)$ is a significant innovation as there are compacta $K,$ for which $A_{D}(K) \not= A(K)$ and replacing $A(K)$ by $A_{D}(K)$ renders correct some incorrect results in the literature (see Section \ref{Ac(K)}). 

Firstly we shall focus our investigation on compacta $K$ which are Cartesian products of planar compact sets   
$$
K = K_{1}\times\cdots\times K_{d},
$$
for which we have information regarding the factors $K_i.$ For such products we can introduce some natural intermediate algebras of functions. 

\begin{defn}
	\label{def:algebras}
	Let $K_{1},\ldots, K_{d}$ be compact subsets of the complex plane $\mathbb C$. For each $i=1,\dots,d$, let $L_{i}\subset \mathbb C\cup\{\infty\}$ be a set containing at least one point from each connected component of $(\mathbb C\cup\{\infty\})\setminus K_{i}$. Also let $K=K_{1}\times\cdots\times K_{d}\subset \mathbb C^{d}$. We denote by $r_{L}(K)$ the set of all complex functions defined on $K$  that can be written as a finite sum of finite products of rational functions  of one variable $z_{i}$ with poles in $L_{i}$, $i=1,\ldots,d$. Let $R_{L}(K)$ be the closure of $r_{L}(K)$ in $C(K)$.  
\end{defn}
 
By abuse of notation  when we write $r_{L}(K)$ or $R_{L}(K)$ we understand that the sets $L_{i}$, $i=1,\ldots,d$, are arbitrary but fixed and they satisfy the above conditions.
 
It follows directly from the definitions that
$$
P(K)\subset R_{L}(K)\subset R(K)\subset\overline{\mathcal O}(K) \subset A_{D}(K)\subset A(K) \subset C(K).
$$
If $K\subset\mathbb C,$ then
\[
	R_{L}(K)= R(K)=\overline{\mathcal O}(K)
\]
by Runge's Theorem. Our main result in this section is that these equalities also hold if $K$ is a product of planar compact sets.

\begin{thm}
	\label{algebrasequivalences}
	For  $K=K_{1}\times \cdots \times K_{d}$ a product of planar compact sets and $L_{i},$ $i=1,\ldots , d,$ as in Definition \ref{def:algebras}, we have 
	\[
		R_{L}(K)= R(K)= \overline{\mathcal O}(K).
	\]
\end{thm}

\begin{proof}
	It suffices to prove that every function $g$ in $\mathcal O(K)$ can be uniformly approximated on $K$ by functions in $r_{L}(K)$. Let $g$ be a holomorphic function on an open set $V_{g}$ with $K\subset V_{g}\subset \mathbb C^{d}$. We can  find open sets $U_{i}$, $i=1,\ldots,d$, such that $K_{i}\subset U_{i}\subset \mathbb C$ and $K\subset U_{1}\times\cdots\times U_{d}\subset V_{g}$.
		
	For $d=1$ the result follows from Runge's Theorem  in the complex plane. We proceed by induction. Suppose   the statement of the theorem holds for $d-1$. For a point $(z_{1},\ldots,z_{d})\in \mathbb C^{d}$ we write $w=(z_{1},\ldots,z_{d-1})$ and $z=z_{d}$. We choose a cycle $\gamma$ in $U_{d}\setminus K_{d}$ such that 
	\[
		g(w,z)=\frac{1}{2\pi i}\int_{\gamma}\frac{g(w,\zeta)}{\zeta-z}\, d\zeta
	\]
	for all $z\in K_{d}$ and $w\in U_{1}\times\cdots\times U_{d-1}$. The cycle $\gamma$ can be chosen to be a finite sum of closed polygons. Since the map $(\zeta, w,z)\mapsto \frac{g(w,\zeta)}{\zeta-z}$ is uniformly continuous on the compact set $\gamma\times K_{1}\times \cdots\times K_{d-1}\times K_{d}$ and the total length of $\gamma$ is finite, the above Cauchy integral can be uniformly approximated by Riemann sums. Thus, for each $\varepsilon >0$  there exists  a natural number $M$, constants $c_{1},\ldots, c_{M}\in \mathbb C$ and complex numbers $\zeta_{1},\ldots,\zeta_{M}$ in $\gamma$ such that
	\[
		\Big\vert g(w,z)-\sum_{k=1}^{M}c_{k} g(w,\zeta_{k})\frac{1}{\zeta_{k}-z}\Big\vert<\frac{\varepsilon}{2}
	\]
	for all $(w,z)\in K$.
		
	By the induction hypothesis, the functions $w\mapsto g(w,\zeta_{k})$,  can be approximated uniformly on $K_{1}\times \cdots\times K_{d-1}$ by functions $g_{k}(w)$ that are finite sums of finite products of functions of one complex variable with poles in the set $L_{i}$, $i=1,\ldots,d-1$ for $k=1,\ldots,M$. Since $\zeta_{k}\in \mathbb C\setminus K_{d}$ for $k=1,\ldots, M$, by Runge's Theorem in one variable, the functions $z\mapsto \frac{1}{\zeta_{k}-z}$ can be uniformly approximated on $K_{d}$ by rational functions of the variable $z=z_{d}$ with poles in $L_{d}$.  Thus, there exists $f_{k}\in r_{L}(K_{d})$, $k=1,\ldots,M$, satisfying 
	\[
		\Big\vert \sum_{k=1}^{M}c_{k}g(w,\zeta_{k})\frac{1}{\zeta_{k}-z}-\sum_{k=1}^{M}c_{k}g_{k}(w)f_{k}(z)\Big\vert <\frac{\varepsilon}{2}
	\]
	for all $(w,z)\in K$. We denote by $f(w,z)=\sum_{k=1}^{M}c_{k}g_{k}(w)f_{k}(z)\in r_{L}(K)$. It follows that
	\[
		\vert g(w,z)-f(w,z)\vert<\varepsilon
	\]
	for all $(w,z)\in K$, which proves the desired statement.
\end{proof}

Since for a product of planar compact sets $K=K_{1}\times \cdots \times K_{d}$ all the algebras $R_{L}(K)$, $R(K)$   and $\overline{\mathcal O}(K)$ coincide, we shall denote any of them simply by $R(K)$. However, for  compact sets $K\subset \mathbb C^{d}$ which are not products of planar compact sets, it is not always true that $R(K)= \overline{\mathcal O}(K)$. For an example see Section \ref{appendix}.

From the previous theorem we obtain that the algebra $R(K_{1}\times \cdots \times K_{d})$ coincides with the tensor product of the algebras $R(K_{1}), \ldots, R(K_{d})$, which we consider as in \cite{MR0249639} to be the closure in $C(K)$ of finite sums of finite products of elements of $R(K_{i})$, $i=1,\ldots,d$. 
\begin{cor}
	For planar compact sets $K_{1}, \ldots, K_{d}$, we have 
	\[
		R(K_{1}\times \cdots \times K_{d})=R(K_{1})\otimes \cdots \otimes R(K_{d}).
	\]
\end{cor}

In Theorem \ref{algebrasequivalences} the allowable singularities of the approximating functions depend strongly on the compact set $K=K_{1}\times\cdots\times K_{d}$.  In the following corollary we show that for products of planar open sets $U=U_{1}\times\cdots\times U_{d}$, we can choose a common set of singularities that works for all compact subsets of $U$. We shall denote by $r_{L}(U)$  the set of finite sums of finite products of rational functions  of one variable $z_{i}$ with poles in a set $L_{i}$, $i=1,\ldots,d$.

\begin{cor}
	\label{thrm:convergenceoncompatsubsets}
	Let $U_{1},\ldots,U_{d}$ be open subsets of $\mathbb C$ and $U=U_{1}\times\cdots\times U_{d}\subset \mathbb C^{d}$. For $i=1,\ldots,d$, let $L_{i}$ be a set containing at least one point from each component of $(\mathbb C\cup\{\infty\})\setminus U_{i}$. For each $f\in \mathcal O(U)$ there exists a sequence $g_{n}\in r_{L}(U)$, $n=1,2,\ldots$ converging to $f$ uniformly on compacta of $U$. Therefore, the set of rational functions   with set of singular points disjoint from $U$ is dense in $\mathcal O(U)$.
\end{cor}

\begin{proof}
	For each $i=1,\ldots, d$, let $K_{i}^{m}$, $m=1,2,\ldots$ be a normal exhaustive sequence of compact subsets of $U_{i}$. We notice that each connected component of $(\mathbb C\cup\{\infty\})\setminus K_{i}^{m}$ contains a connected component of $(\mathbb C\cup\{\infty\})\setminus U_{i}$. Thus, the set $L_{i}$ contains at least one point from each connected component of $(\mathbb C\cup\{\infty\})\setminus K_{i}^{m}$. We set $K^{m}=K_{1}^{m}\times\cdots\times K_{d}^{m}$. Then every compact subset of $U$ is contained in some $K^{m}$. Therefore, it suffices to find $g_{m}\in r_{L}(K^{m})=r_{L}(U)$ such that 
	\[
		\sup_{z\in K^{m}}\vert g_{m}(z)-f(z)\vert < \frac{1}{m}.
	\] 
	By Theorem \ref{algebrasequivalences} applied to the product  $K^{m}$ such a $g_{m}$ exists. This completes the proof.
\end{proof}

\begin{prop}
	\label{propopencompacts}
	Let $K_{i}$, $i=1,\ldots,d$, be compact sets in $\mathbb C$ and $K=K_{1}\times\cdots\times K_{d}$. Using the notation of Theorem \ref{algebrasequivalences}, if $f\in \mathcal{O}(K)$, $\varepsilon>0$ and $I$ is a finite subset of $\{0,1,2,\ldots\}$, then there exists $g\in r_{L}(K)$ such that for every $(\alpha_1,\ldots,\alpha_d)\in I^d$ we have
	\[
		\sup_{z\in K}\Big\vert  \frac{\partial^{\alpha_1+\cdots+\alpha_d}}{\partial z_1^{\alpha_1}\cdots\partial z_d^{\alpha_d}}(g-f)(z)\Big\vert <\varepsilon.
	\]
\end{prop}

\begin{proof}
	The function $f$ is holomorphic on an open set $V_{f}$ such that $K\subset V_{f}\subset \mathbb C^d$. We can find open sets $U_i$ with $K_i\subset U_i\subset \mathbb C$ such that the open set $U=U_1\times \cdots\times U_d$ satisfies $K\subset U\subset V_{f}$. Thus $f\in \mathcal O(U)$. According to Corollary \ref{thrm:convergenceoncompatsubsets} there exists a sequence $g_m\in r_{L}(U)$, $m=1,2,\ldots$ converging to $f$ uniformly on compacta of $U$.
	Since uniform convergence of holomorphic functions on an open set $U$ implies uniform convergence of partial derivatives on compact subsets of $U$, we have that for every $(\alpha_1,\ldots,\alpha_d)\in I^d$ 
	\[
		\sup_{z\in K}\Big\vert  \frac{\partial^{\alpha_1+\cdots+\alpha_d}}{\partial z_1^{\alpha_1}\cdots\partial z_d^{\alpha_d}}(g_{m}-f)\Big\vert\to 0
	\] 
	as $m\to\infty$. It suffices to set $g=g_m$ for sufficiently large  $m$. This concludes the proof.
\end{proof}

Let $g=(g_m)_{m=1}^{\infty}$  be a sequence in $ \mathcal O(K)$  such that for every $\alpha=(\alpha_1,\ldots,\alpha_d)\in \{0,1,\ldots\}^d$, the corresponding sequence of  mixed partial derivatives $\frac{\partial^{\alpha_1+\cdots+\alpha_d}}{\partial z_1^{\alpha_1}\cdots\partial z_d^{\alpha_d}}g_m$ converges uniformly on $K$ to some function $w_{\alpha,g}$ that depends on $\alpha$ and the sequence $g=(g_m)_{m=1}^\infty$. We denote by $R^\infty(K)$ the set of families
\[
	\{w_{\alpha,g}:\alpha\in\{0,1,\ldots\}^d \text{ for } g=(g_m)_{m=1}^\infty\text{ in } \mathcal O(K) \text{ as above}\}.
\]

Proposition \ref{propopencompacts} implies that in the definition of $R^{\infty}(K)$ it is sufficient to consider only sequences $g=(g_n)_{n=1}^{\infty}$ where $g_{n}\in r_{L}(K)$. If $K_i$ are regular closed then  $R^\infty(K)$ can be identified with a set of holomorphic functions on $K_1^\circ\times\cdots\times K_d^\circ$, because $w_{\alpha,g}=\frac{\partial^{\alpha_1+\cdots+\alpha_d}}{\partial z_1^{\alpha_1}\cdots\partial z_d^{\alpha_d}}w$ on $K_1^\circ\times\cdots\times K_d^\circ$, where $w=\lim_{n\to \infty}g_n$ and all these functions extend continuously on $K$. Then, $r_{L}(K)$ is dense in $R^\infty(K)$ endowed with the seminorms
\[
	\sup_{z\in K^\circ}\Big\vert \frac{\partial^{\alpha_1+\cdots+\alpha_d}}{\partial z_1^{\alpha_1}\cdots\partial z_d^{\alpha_d}}w(z)\Big\vert,\hspace{8mm}\text{for } (\alpha_1,\ldots,\alpha_d)\in\{0,1,2,\ldots\}^d
\]
 provided that $K_i$ is regular closed for $i=1,\ldots,d$.

We close this section replacing rational approximation by approximation by polynomials, see \cite{Gauthier2014}. For an open set $U \subset \mathbb C^d$ we denote by $P(U)$ the class of functions $f:U\to \mathbb C$ such that there exists a sequence $(P_n)_{n=1}^\infty$ of polynomials of $d$ complex variables converging to $f$ uniformly on each compact subset of $U$.

\begin{prop}
	\label{aequalsb}
	Let $K_i \subset \mathbb C$, $i=1,\ldots,d$, be compact sets and $K=K_1 \times \cdots \times K_d$. Then the following are equivalent:
	\begin{itemize}
		\item[(a)] $P(K) = \overline{\mathcal{O}}(K)$;
		\item[(b)] For every $i=1,\ldots,d$ the set $\mathbb C \setminus K_i$ is connected.
	\end{itemize}
\end{prop}

\begin{proof}
	To show that $(b)$ implies $(a)$ it suffices to set $L_i=\{\infty\}$ and apply Theorem \ref{algebrasequivalences}, because every rational function of one variable with poles contained in $\{\infty\}$ is a polynomial of one variable.
		
	To see that $(a)$ implies $(b)$, let $i_0\in \{1,\ldots,d\}$ be fixed. Suppose $(a)$ and suppose that $\mathbb C\setminus K_{i_0}$ is not connected. We can pick a point $b\in \mathbb C$  in a bounded connected component $V$ of $\mathbb C\setminus K_{i_0}$. Then the function $f(z_1,\ldots, z_d)=\frac{1}{z_{i_0}-b}$ belongs to $\overline{\mathcal O}(K)$ and so is a uniform limit on $K$ of polynomials $(P_n(z_1,\ldots,z_d))_{n=1}^\infty$. We fix $z_i=w_i\in K_i$ for all $i\in\{1,\ldots,d\}\setminus\{i_0\}$. Then the sequence of polynomials in one variable, $Q_n(z)=P_n(z_1(z),\ldots,z_d(z))$, where $z_{i_0}(z)=z$ and $z_i(z)=w_i$ for all $i\ne i_0$, converges uniformly on $K_{i_0}$ to $\frac{1}{z_{i_0}-b}$.  
	Fix a natural number $n_0$ so that 
	\[
		\max_{z_{i_0}\in K_{i_0}}\Big\vert Q_{n_0}(z_{i_0}) - \frac{1}{z_{i_0}-b}\Big\vert<\frac{1}{2\max_{z_{i_0}\in \partial V}\vert z_{i_0}-b\vert}.
	\] Then, by the maximum modulus principle on $V$,
	\begin{align*}
		1/2 & >\max_{z_{i_0}\in \partial V}\vert z_{i_0}-b\vert\cdot \max_{z_{i_0}\in\partial V}\Big\vert Q_{n_0}(z_{i_0}) - \frac{1}{z_{i_0}-b}\Big\vert \\
		    & \geq\max_{z_{i_0}\in \partial V}\vert (z_{i_0}-b) Q_{n_0}(z_{i_0}) -1\vert                                                                  \\
		    & \geq\max_{z_{i_0}\in \overline{V}}\vert (z_{i_0}-b) Q_{n_0}(z_{i_0}) -1\vert                                                                \\
		    & \geq\vert (b-b) Q_{n_0}(b) -1\vert\geq 1,                                                                                                   
	\end{align*}
	which is a contradiction. Thus, $\mathbb C \setminus K_{i_{0}}$ is connected and the proof is complete.
		
\end{proof}

Arguing in a similar way we obtain the following:

\begin{prop}
	Let $U_i \subset \mathbb C$, $i=1,\ldots,d$,  be open sets and $U=U_1\times\cdots\times U_d$. Then the following are equivalent:
	\begin{itemize}
		\item[(a)] $P(U) = \mathcal O(U)$;
		\item[(b)] For every $i=1,\ldots,d$, the set $U_i$ is simply connected.
	\end{itemize}
\end{prop}

As a direct consequence of Proposition \ref{propopencompacts} we obtain the following special case for polynomials.
\begin{prop}
	Let $K_{i}$, $i=1,\ldots,d$,  be compact sets in $\mathbb C$ such that $\mathbb C\setminus K_i$ is connected. Let $f$ be a function holomorphic on an open set containing $K=K_1\times\cdots\times K_d$, let $\varepsilon>0$ and $I$ a finite subset of $\{0,1,\ldots\}^d$. Then there exists a polynomial $Q(z_1,\ldots,z_d)$ such that
	\[
		\sup_{z\in K}\Big\vert  \frac{\partial^{\alpha_1+\cdots+\alpha_d}}{\partial z_1^{\alpha_1}\cdots\partial z_d^{\alpha_d}}(f-Q)(z)\Big\vert <\varepsilon
	\]
	for every $\alpha=(\alpha_1,\ldots,\alpha_d)\in I$.
\end{prop}
 
\section{The algebra $A_{D}(K)$}
\label{Ac(K)}

Let $K\subset \mathbb C^{d}$ be a compact set.  As we already mentioned $P(K)\subset \overline{\mathcal{O}}(K)\subset A(K)$. One of the most challenging questions in approximation theory is to determine under which conditions on $K$, we have equalities. The situation in one complex variable is well understood. In particular, let $K$ be a compact set in $\mathbb C$. The equality $P(K)=\overline{\mathcal O}(K)$ holds if and only if the complement of $K$ is connected. The  same characterization holds for the equality $P(K)=A(K)$, according to the celebrated theorem of Mergelyan.  Vitushkin \cite{MR0229838} gave the best Mergelyan type theorem by providing necessary and sufficient conditions for the equality $\overline{\mathcal O}(K)=A(K)$. However, in several complex variables the analogous theory is much less developed, see \cite{FFW,MR2276419}. It is natural to investigate the case where $K$ is a product of planar compact sets. For example, in Proposition  \ref{algebrasequivalences} we saw that if $P(K_{i})=\overline{\mathcal O}(K_{i})$ then $P(K_{1}\times\cdots\times K_{d})=\overline{\mathcal O}(K_{1}\times\cdots\times K_{d})$.


The following question mentioned in  \cite[Section 6]{MR0249639} naturally arises.
\begin{quest}\label{GG}
	If $K_{i}\subset\mathbb{C}$ are compact sets such that $\overline{\mathcal {O}}(K_i)=A(K_i)$ for $i=1,2,\dots , d$, is it always true that $\overline{\mathcal{O}}(K_{1}\times\cdots \times K_{d})=A(K_{1}\times\cdots \times K_{d})$? 
\end{quest}

The example below provides a negative answer to this question.

\begin{ex}\label{counterexample1}
	Let $K_{1}=\{0\}\subset\mathbb{C}$ and $K_2=\{w\in\mathbb{C}: |w|\leq 1\}\subset \mathbb{C}$ and let $g: K_{2}\to\mathbb{C}$ be a continuous function which is not holomorphic on $\{w\in\mathbb{C}: |w|< 1\}$; for instance take $g(w)=|w|$. Let $f: K_{1}\times K_{2}\to \mathbb{C}$ be defined by $f(z,w)=g(w)$ for all $(z,w)\in K_1\times K_2$.  Since the interior of $K_1\times K_2$ in $\mathbb{C}^2$ is void, we have $A(K_1\times K_2)=C(K_1\times K_2)$. It follows that $f\in A(K_1\times K_2)$. Moreover, we have that $A(K_1)=\overline{\mathcal {O}}(K_1)=\mathbb{C}$ and $A(K_2)=\overline{\mathcal {O}}(K_2)$. An affirmative answer to Question \ref{GG} would then imply that $\overline{\mathcal{O}}(K_{1}\times K_{2})=A(K_{1}\times K_{2})$, and so $f\in \overline{\mathcal{O}}(K_{1}\times K_{2})$. Therefore, there would be a sequence  of functions $g_n$ holomorphic on open sets $V_n$ with $K_{1}\times K_{2}\subset V_n\subset \mathbb{C}^2$ such that 
	\[
		\sup _{(z,w)\in K_{1}\times K_{2}} |g_{n}(z,w)-f(z,w)|\to 0, \text{ as } n\to\infty .
	\]
	This implies that $g(w)=f(0,w)$ is holomorphic on  $\{w\in\mathbb{C}: |w|< 1\}$, which contradicts our assumption concerning the function $g$. 
Therefore the answer to Question \ref{GG} is negative.
		
\end{ex}

Our goal is to  replace the algebra $A(K)$ by an appropriate algebra so that the answer to Question \ref{GG} becomes positive. More generally, we would like to obtain Mergelyan type theorems for other compact sets in $\mathbb C^{d}$. We recall that $\overline{\mathcal{O}}(K)$ is the closure in $C(K)$ of functions holomorphic on (varying) open neighbourhoods of $K$. Thus, we shall consider sequences of functions $g_n$ holomorphic on open sets $V_n$ with $K\subset V_{n}\subset \mathbb{C}^d$ and converging uniformly on $K$ to some function $f: K\to\mathbb{C}$. We shall investigate some necessary properties of the limit function $f$. The definition of the new class $A_{D}(K)$ that we shall introduce will incorporate these necessary properties.

Firstly, each $g_n$ is continuous on $K$, therefore the uniform limit $f$ should be continuous on $K$. This will be the first requirement. Secondly, let $m$ be a natural number and $D$ an open polydisc in $\mathbb{C}^{m}$ (or more generally any open subset of $\mathbb{C}^m$). Let $\phi : D\to K\subset \mathbb{C}^{d}$ be a holomorphic mapping. Then, since $K\subset V_n \subset \mathbb{C}^d$ and $g_n$ is holomorphic on $V_n$, it follows that the composition $g_n\circ \phi: D\to\mathbb{C}$ is well defined and holomorphic on $D$. Moreover, $g_n\to f$ uniformly on $K$. It follows that $g_n\circ\phi\to f\circ \phi$ uniformly on $D$, where $f\circ\phi$ is a well defined function because $\phi (D)\subset K$ and $f$ is defined on $K$. Since $f\circ \phi$ is a uniform limit of holomorphic functions on the open set $D$, it follows that $f\circ \phi$ is holomorphic on $D$. This will be the second requirement.

In view of  Hartogs' Theorem on separate holomorphicity it suffices to restrict to the case $m=1$, for open discs $D$  in $\mathbb{C}$. Furthermore, we shall show that it is enough that our requirement be valid for mappings $\phi :D\to K\subset \mathbb{C}^d$ holomorphic and injective on open discs $D\subset\mathbb{C}$. Indeed, let $\phi : D\to K\subset \mathbb{C}^{d}$ be a holomorphic mapping, not necessarily injective. We must show that $f\circ\phi$ is holomorphic around every $z_0\in D$. We have $\phi(z)=(\phi _{1}(z), \dots, \phi _{d}(z))\in K\subset{\mathbb{C}^{d}}$, where $\phi _{i}(z)\in \mathbb{C}$. If there is $r>0$ such that $\phi$ restricted to the disc $D(z_{0},r)=\{z\in\mathbb{C}: |z-z_{0}|<r\}\subset D$ is injective, then, by our assumption concerning injective mappings, it follows that $f\circ \phi$ is holomorphic on 
$D(z_{0},r)$. Suppose that there is no $r>0$ such that $\phi=(\phi _{1},\dots, \phi _{d})$ is injective on $D(z_{0},r)$. Then $\phi _{1} '(z_{0})=\dots= \phi _{d} '(z_{0})=0$. If at least one of the functions $\phi _1 , \dots , \phi _{d}$, say  $\phi _{j}$, is non-constant, then $\phi _{j}'\not\equiv 0$, so its zero set $Z$ is a set of isolated points. Thus $\phi _{j}$ is locally injective on $D\setminus Z$ which implies that $\phi$ is locally injective on $D\setminus Z$. Thus, by our assumption, we conclude that $f\circ \phi$ is holomorphic on $D\setminus Z$. Since $f\circ \phi$ is continuous on $D$, it follows that it is holomorphic on $D$. If all functions $\phi _1 , \dots , \phi _{d}$ are constants, then $f\circ \phi$ is constant, hence holomorphic on $D$. 

These considerations suggest that we introduce the following definition.

\begin{defn}\label{A new}
	Let $K\subset\mathbb{C}^d$ be a compact set (or more generally a closed set). A function $f: K\to\mathbb{C}$ is said to belong to the class $A_{D}(K)$ if it is continuous on $K$ and, for every open disc $D\subset\mathbb{C}$ and every injective mapping $\phi : D\to K\subset\mathbb{C}^d$ holomorphic on $D$, the composition $f\circ \phi: D\to \mathbb{C}$ is holomorphic on $D$.
\end{defn}

Roughly speaking the above definition tells us that $f$ should be continuous on $K$ and holomorphic on every complex manifold contained in $K$. In fact, it is enough to restrict our attention to injective holomorphic mappings $\phi$ defined on the open unit disc. Moreover, one can check that in the above definition we may replace  the open discs $D$ by closed discs $\overline D$ and $\phi$ by  injective mappings that are continuous  on $\overline D $ and holomorphic on $D$.  

\begin{rem}
	Note that for compact sets $K\subset \mathbb C^{d}$ we have that $\overline{\mathcal O}(K)\subset A_{D}(K)\subset A(K)$ and  if $K$ is a planar compact set then $A_{D}(K)=A(K)$.
\end{rem}

In this paper we are mostly interested in products of planar compact sets. In this situation we have the following characterization of $A_{D}(K)$. 

\begin{prop}
	\label{M}
	Let $K_{i}\subset\mathbb{C}$ be compact sets, $i=1,2,\dots, d$. The following are equivalent:
	\begin{itemize}
		\item[(a)] $f\in A_{D}(K_{1}\times\cdots \times K_{d})$;
		\item[(b)]  $f\in C(K_{1}\times\cdots \times K_{d})$ and for each $i_{0}\in\{1,\dots,d\}$ and each $w_i\in K_i$, $i\in\{1,\dots,d\}\setminus\{i_{0}\}$, the function $K_{i_{0}}^{\circ}\ni z \mapsto f(z_{1}(z),\dots, z_d(z))\in\mathbb{C}$ is holomorphic on $K_{i_{0}}^{\circ}$, where $z_{i_{0}}(z)=z$ and   $z_{i}(z)=w_i$ for $i\in\{1,\dots,d\}\setminus\{i_{0}\}$.
	\end{itemize}
\end{prop}

\begin{proof}
	It is straightforward  that $(a)$ implies $(b)$. For the converse, assume that $f$ is as in $(b$). Let $K=K_{1}\times\cdots\times K_{d}$. By Hartogs' Theorem $f$ is holomorphic on $K^{\circ}$. Consider $\phi : D\to K \subset \mathbb{C}^{d}$  an injective holomorphic mapping on an open disc $D\subset\mathbb{C}$, and let $z_{0}\in D$. If $\phi (z_{0})\in K^{\circ}$, then the composition $f\circ\phi$ is holomorphic on an open set containing $z_{0}$. If $\phi (z_{0})\in\partial K$, where $\phi=(\phi _{1}, \ldots, \phi _{d})$, then for at least one $i\in\{1,\dots , d\}$ we have $\phi _{i}(z_{0})\in\partial K_{i}$. The open mapping theorem for holomorphic functions implies that $\phi _{i}\equiv c_{i}$ for some constant $c_{i}\in\mathbb{C}$. Let $I$ denote the set of all these indices $i$ for which $\phi_{i}(z_{0})=c_{i}\in\partial K_i$ and let $J=\{1,\dots , d\}\setminus I$. Then $f(\phi _{1}(z),\dots , \phi _{d}(z))=g\circ\tilde{\phi}(z)$, where $\tilde{\phi}: D\to\prod _{i\in J} K_i$ is defined by $\tilde{\phi}=(\phi _{i})_{i\in J}$ and $g((w_i)_{i\in  J})=f((z_{i})_{i\in\{1,\dots,d\}})$, where $z_{i}=w_{i}$ for $i\in J$ and $z_{i}=c_{i}$ for $i\in I$. According to the assumption in ($b$)  on separate analyticity of $f$ and using Hartogs' Theorem, the function $g: \prod _{i\in J}K_i\to \mathbb{C}$ is holomorphic on the interior of $\prod _{i\in J}K_i$ and $\tilde{\phi}(z_{0})$ belongs to this set. Since $\tilde{\phi}$ is holomorphic it follows that $f\circ\phi =g\circ \tilde{\phi}$ is holomorphic around $z_{0}$, provided $J\neq \emptyset$. If $J=\emptyset$ then $\phi$ and $f\circ\phi$ are constants, and hence they  are holomorphic. This proves ($a$).
		
\end{proof}

Condition ($b$) in the above proposition is motivated by \cite{2013arXiv1304.5511M}, where the set of uniform limits of polynomials on a product of closed discs is characterized. We notice that in this condition it is essential that the points $w_i$ may belong to the boundary $\partial K_{i}$ and not necessarily to  $K_{i}^{\circ}$. It follows that the function $f$ in the Example \ref{counterexample1} does not belong to $A_D(K_1\times K_2)$. Thus there is hope that Question \ref{GG} could have an affirmative answer, if we replace $A(K_{1}\times\cdots \times K_{d})$ with the new algebra of functions $A_{D}(K_{1}\times\cdots \times K_{d})$ introduced in Definition  \ref{A new}. In fact we shall prove that this is true in Theorem \ref{mainresult}.

In connection  with \cite{FalcoNestoridis2017,FalcoNestoridis2018,MR0249639,Gauthier2014}, we notice the following.

\begin{cor}
	If $K_{i}\subset\mathbb{C}$ are planar compact  regular closed sets for $i=1,\ldots,d$, then
	\[
		A(K_{1}\times\cdots\times K_{d})=A_{D}(K_{1}\times\cdots\times K_{d}). 
	\]
\end{cor}

\begin{proof} Let $K=K_{1}\times\cdots\times K_{d}$. Obviously $A_{D}(K)$ is a subalgebra of $A(K)$. For the other inclusion, assume that  $f\in A(K)$. It suffices to show that $f$ satisfies condition $(b)$ in Proposition \ref{M}.  Every $w_i\in K_i$ can be approximated by a sequence $w_{i,n}\in K_{i}^{\circ}$, $n=1,2,\dots$ and the function $K_{i_{0}}^{\circ}\ni z \mapsto f(z_{1}(z),\dots, z_d(z))\in\mathbb{C}$ is a uniform limit of functions which are holomorphic on $K_{i_{0}}^{\circ}$. Here the points $w_i$ have been replaced by the points $w_{i,n}\in K_{i_{0}}^{\circ}$ where $z_{i_{0}}(z)=z$ and   $z_{i}(z)=w_i$ for $i\ne i_{0}$.
\end{proof}

We mention that an advantage of condition $(b)$ in  Proposition \ref{M} is that it can be naturally extended to infinite products, see Definition \ref{infiniteprod}. It suffices to require that the function $f: \prod _{i\in I} K_{i}\to\mathbb{C}$ be continuous when $\prod _{i\in I} K_{i}$ is endowed with the product topology and that $f$ be separately holomorphic on $K_{i_{0}}^{\circ}$ when all other variables are fixed in $K_{i}$ (even in $\partial K_{i}$).

\section{Mergelyan type theorems}
\label{Mergelyanresult}
In this section we prove two Mergelyan type theorems in $\mathbb C^{d}$ using the algebra $A_{D}(K)$. Their analogues in general fail for $A(K)$. A direct corollary of our first main result, Theorem \ref{mainresult}, provides an affirmative answer to Question \ref{GG}   for the algebra $A_{D}(K)$.  Theorem \ref{mainresult} is the analogue of \cite[Corollary 6.4]{MR0249639} for this new algebra. Our second main result, Theorem \ref{counterexample2},  deals with compact sets $K$ which are graphs of functions.

Before presenting our main theorems we need some auxiliary results. Let $K$ be a planar compact set and $W$ a Banach space. Let $C(K,W)$ denote the Banach space of continuous functions from $K$ to $W$, endowed with the supremum norm. We also denote by $W^{*}$ the dual space of $W$. Let $W$ be a closed subspace of $C(Y)$, where $Y$ is a compact space. It is natural to identify $C\big(K, C(Y)\big)$ with $C(K\times Y)$ via the bijection
\begin{equation}
	\label{identification}
	\begin{array}{ccl}
		C(K\times Y)   & \longrightarrow & C\big(K,C(Y)\big)    \\
		F(\cdot,\cdot) & \longmapsto     & x\mapsto F(x,\cdot). 
	\end{array}
\end{equation}
Therefore, $C(K,W)$ is identified with a  subspace of $C(K\times Y)$.

Let $A(K,W)$ be the subalgebra of $C(K,W)$ of functions holomorphic on $K^\circ$. Let $\overline{\mathcal{O}}(K,W)$ be the closure in $C(K,W)$ of the functions which extend holomorphically in a neighborhood of $K$. The following two results follow from \cite[Theorem 6.1]{MR0249639}.

\begin{thm}[{\cite{MR0249639}}]
	\label{thrmGG}
	Let $K$ be a planar compact set. If $f\in C(K,W)$, the following are equivalent:
	\begin{itemize}
		\item[(a)] $f\in \overline{\mathcal O}(K,W)$;
		\item[(b)] For each $L\in  W^*$ we have $L\circ f\in \overline{\mathcal O}(K)$.
	\end{itemize}
\end{thm}

\begin{cor}[{\cite[Corollary 6.2]{MR0249639}}]
	\label{cor52}
	Let $K\subset \mathbb C$ be compact. If $\overline{\mathcal O}(K)=A(K)$, then $\overline{\mathcal O}(K,W)=A(K,W)$ for all Banach spaces $W$.
\end{cor}

\begin{lem}
	\label{claim1}
	Let $K_1\subset\mathbb C$ and $K_2\subset\mathbb C^m$ be compact sets. Then with the identification in \eqref{identification} we have $\overline{\mathcal O}\big(K_1,\overline{\mathcal O}(K_2)\big)\subset \overline{\mathcal O}(K_1\times K_2).$
\end{lem}

\begin{proof}
	Let $f\in \overline{\mathcal O}\big(K_1,\overline{\mathcal O}(K_2)\big)$ and $\varepsilon>0$. There is an open neighbourhood $V_{f}$ of $K_1$ and a holomorphic function $\tilde f:V_{f}\to \overline{\mathcal O}(K_2)$ such that 
	\[
		\Vert  f(z) - \tilde f(z)\Vert_{K_2}<\varepsilon/3
	\] for every $z\in K_{1}$. Then, there is a cycle $\gamma$ in $V_{f}\setminus K_1$ for which we have by the Cauchy formula:
	\[
		\tilde f(z)=\frac{1}{2\pi i}\int_\gamma \frac{\tilde f(\zeta)}{\zeta-z}\, d\zeta,
	\]
	for all $z\in K_1$.
		
	We observe that the function $(\zeta,z)\mapsto \frac{\tilde f(\zeta)}{\zeta-z}$ defined on $\gamma\times K_1$ is uniformly continuous, as it is a continuous function defined on a compact set, so we can approximate   the integral uniformly on $K_1$ by its Riemann sums
	\[
		\Big\Vert \frac{1}{2\pi i}\int_\gamma \frac{\tilde f(\zeta)}{\zeta-z}\, d\zeta - \sum_{l=1}^M a_l\frac{\tilde f(\zeta_l)}{\zeta_l-z}\Big\Vert_{K_{2}}<\varepsilon/3,
	\]
	for every $z\in K_{1}$, where $a_{l}\in\mathbb C$, $\zeta_{l}\in \gamma$ and $f(\zeta_{l}) \in\overline{\mathcal O}(K_{2})$ for $l = 1,\ldots ,M$. Thus, for each $l=1,\ldots, M$ there is a neighbourhood $U_l$ of $K_2$ and a holomorphic  function $h_l:U_l\to \mathbb C$ such that $\Vert h_l-\tilde f(\zeta_l)\Vert_{K_2}<\varepsilon_l$.  We may choose $\varepsilon_l$ sufficiently small such that
	\[
		\Big\Vert \sum_{l=1}^M a_l\frac{\tilde f(\zeta_l)}{\zeta_l-z}  -  \sum_{l=1}^M a_l\frac{h_l}{\zeta_l-z}\Big\Vert_{K_{2}} < \varepsilon/3
	\]
	for every $z\in K_{1}$, and so, by the triangle inequality,
	\[
		\Big\Vert f(z) - \sum_{l=1}^M a_l\frac{h_l}{\zeta_l-z}\Big\Vert_{K_{2}}
		< \varepsilon.
	\]
	for every $z\in K_{1}$. Set $U=U_1\cap\cdots\cap U_{d}$ which is also an open neighbourhood of $K_2$.
		
	Now, notice that the function $(z,w)\mapsto \sum_{l=1}^M a_l\frac{h_l(w)}{\zeta_l-z}$
	is holomorphic on $V_{f}\times U$ which is a neighbourhood of $K_1\times K_2$, and yields the desired approximation.
\end{proof}

\begin{lem}
	\label{claim2}Let $K_1\subset \mathbb C$ and $K_2\subset\mathbb C^m$ be compact sets. Then with the identification in \eqref{identification} we have  $A_D(K_1\times K_2)\subset A\big(K_1,A_{D}(K_2)\big)$.
\end{lem}

\begin{proof}
	We need to show that the   identification \eqref{identification} maps $A_{D}(K_1\times K_2)$ to $ A\big(K_1,A_{D}(K_{2})\big)$, that is the mapping
	\begin{equation*}
		\begin{array}{ccl}
			A_{D}(K_1\times K_2) & \longrightarrow & A\big(K_1,A_{D}(K_{2})\big) \\
			F(\cdot,\cdot)       & \longmapsto     & x\mapsto F(x,\cdot)         
		\end{array}
	\end{equation*}
	is well defined.

	We need to show that:
	\begin{itemize}
		\item[(1)] $F(x,\cdot)\in C(K_2)$ for all $x\in K_1$;
		\item[(2)] For every  $x\in K_1$ and $\Phi:D\to K_2$ holomorphic and injective on a disc $D$ in $\mathbb C$ we have that $F(x,\cdot)\circ \Phi$ is holomorphic on $ D$;
		\item[(3)] The mapping $x\mapsto F(x,\cdot)$ belongs to $C\big(K_1,A_D(K_2)\big)$ and is holomorphic on $K_1^\circ$.
	\end{itemize}
		
	Part $(1)$ follows trivially since $F\in C(K_1\times K_2)$.
	For part $(2)$ consider $\Phi:D \to K_2$ holomorphic and injective on an open disc $D$ in $\mathbb C$ and, for each $x\in K_{1}$, the function $\Phi_x:D\to K_1\times K_2$ defined by $\Phi_x(z)=(x,\Phi(z)).$ Since $\Phi_x$ is holomorphic and injective on $D$, and since $F\in A_D(K_1\times K_2)$, we get that $F\circ \Phi_x$ is holomorphic on $D$ for each $x$ in $K_{1}$, by Definition \ref{A new}. We have $(F\circ \Phi_x)(z)=F(x,\Phi(z))=(F(x,\cdot)\circ\Phi)(z)$, which shows that $F(x,\cdot)\circ\Phi$ is holomorphic on $D$ as desired. Thus, $F(x,\cdot)$ belongs to $A_D(K_2)$.
		
	Now we prove $(3)$. The continuity follows  because $F$ is continuous on $K_1\times K_2$. We need to show that for all $L\in A_D(K_2)^*$ we have that 
	\begin{equation*}
		\begin{array}{rcccl}
			K_{1} & \longrightarrow & A_{D}(K_{2}) & \longrightarrow & \mathbb C             \\
			x     & \mapsto         & F(x,\cdot)   & \mapsto         & L\big(F(x,\cdot)\big) 
		\end{array}
	\end{equation*}
	is holomorphic on $K_{1}^{\circ}$.

	By the Hahn-Banach Theorem, the continuous functional $L$ can be extended to $C(K_2)$ and by the Riesz Representation Theorem there is a finite complex measure $\mu$ depending only on $L$ and supported on $K_2$ such that, for all $g\in A_{D}(K_2)$,
	\[
		L(g)=\int_{K_2}g(w)\, d\mu(w).
	\]
	In particular, we have $L\big(F(x,\cdot)\big)=\int_{K_2}F(x,w)\, d\mu(w)$ for each $x\in K_1$. Since $F$ is uniformly continuous on $K_1\times K_2$, the last integral can be approximated uniformly on $K_1$ by ``Riemann sums'' $\sum_{l=1}^M c_l F(x,w_l)$, where $M$ is a natural number, $w_{l}\in K_{2}$ and $c_l\in \mathbb C$ for $l=1,\ldots,M$.
		
	The function $K_1^\circ\ni x\mapsto c_lF(x,w_l)$ is holomorphic on $K_1^\circ$ because $F\in A_D(K_1\times K_2)$. It follows that the function $K_1^\circ\ni x\mapsto \sum_{l=1}^M c_l F(x,w_l)$ is holomorphic on $K_1^\circ$. Therefore,  on $K_1^\circ$, the function $x\mapsto L\big(F(x,\cdot)\big)$, which is the uniform limit of these functions, is also holomorphic. This proves that $x\mapsto F(x,\cdot)$ is holomorphic on $K_{1}^{\circ}$, which completes the proof.
\end{proof}
\begin{rem}
	We note that the previous lemma does not hold if we replace the algebras $A_{D}$ by the algebra $A$ (see Example \ref{counterexample1}).
\end{rem}

\begin{thm}
	\label{mainresult}Let $K_1\subset \mathbb C$, $K_2\subset \mathbb C^m$ be compact sets such that $\overline{\mathcal{O}}(K_1)=A_D(K_1)$ and $\overline{\mathcal{O}}(K_2)=A_D(K_2)$. Then 
	\[
		\overline{\mathcal{O}}(K_1\times K_2)=A_D(K_1\times K_2).
	\]
\end{thm}

\begin{proof}
	By Lemma \ref{claim2} and the assumption that $A_D(K_2)=\overline{\mathcal{O}}(K_2)$ we have
	\begin{equation}
		\label{cor57claim1}
		A_{D}(K_1\times K_2)\subset A(K_1,A_D(K_2))=A(K_1,\overline{\mathcal{O}}(K_2)).
	\end{equation} 
	Since $A(K_{1})=A_D(K_1)=\overline{\mathcal{O}}(K_1)$ by applying Corollary \ref{cor52} with $W=\overline{\mathcal{O}}(K_2)$, we get that 
	\begin{equation}
		\label{cor57claim2}
		A(K_1,\overline{\mathcal{O}}(K_2))=\overline{\mathcal{O}}(K_1,\overline{\mathcal{O}}(K_2)).
	\end{equation}
	Finally, by Lemma \ref{claim1}, we conclude that 
	\begin{equation}
		\label{cor57claim3}
		\overline{\mathcal{O}}(K_1,\overline{\mathcal{O}}(K_2))\subset \overline{\mathcal{O}}(K_{1}\times K_2).
	\end{equation}
	By combining equations \eqref{cor57claim1}, \eqref{cor57claim2} and \eqref{cor57claim3} we obtain that $A_{D}(K_1\times K_2)\subset \overline{\mathcal{O}}(K_1\times K_2)$ and since $\overline{\mathcal{O}}(K_1\times K_2)\subset A_D(K_1\times K_2)$ the desired result follows.
\end{proof}

An immediate consequence of the previous theorem is the following corollary which gives a positive answer to Question \ref{GG} provided we use the algebra $A_{D}(K_{1}\times\cdots\times K_{d})$ instead of $A(K_{1}\times\cdots\times K_{d})$.
\begin{cor}
	\label{answerquestion}
	Let $K_i\subset \mathbb C$, $i=1,\ldots,d$, be compact sets such that $\overline{\mathcal{O}}(K_i)=A_D(K_i)$. Then
	\[
		R_{L}(K_1\times \cdots \times K_{d})=R(K_1\times \cdots \times K_{d})=\overline{\mathcal{O}}(K_1\times \cdots \times K_{d})=A_D(K_1\times \cdots \times K_{d}).
	\]
\end{cor}

We now  give our second main result  where $K$ is not necessarily a product.

\begin{thm}\label{counterexample2}
	Let $K_{1}\subset\mathbb{C}^m$ be a compact set such that $A_{D}(K_1)=\overline{\mathcal O}(K_{1})$ and $\omega  \in A_{D}(K_1)$.  Let $K$ be the graph of the function $\omega$; that is, 
	\[
		K=\{(z,w)\in \mathbb{C}^{m+1}: z\in K_{1}, w=\omega (z) \}.
	\]
	Then $A_{D}(K)=\overline{\mathcal O}(K)$.
\end{thm}
\begin{proof}
	It suffices to show $A_{D}(K)\subset\overline{\mathcal O}(K)$ since the other inclusion follows trivially. Let $f\in A_{D}(K)$ and let $\Psi : K_{1} \to K$ be the mapping defined by $\Psi (z)=(z,\omega (z))$. We claim that $f\circ\Psi\in A_{D}(K_1)$. First, we note that $f\circ\Psi\in C(K_1)$. Let $\Phi : D\to K_1$ be a holomorphic and injective mapping on some disc $D\subset\mathbb{C}$. We shall show that $(f\circ\Psi)\circ\Phi$ is holomorphic on $D$. To see this, we notice that $\Psi\circ\Phi: D\to K$ is injective and holomorphic on $D$, since $\Psi\circ\Phi (z)=(\Phi (z), \omega (\Phi(z)))$, the mapping $\Phi$ is holomorphic and injective on $D$ and $\omega\in A_{D}(K_1)$. Thus, $(f\circ\Psi)\circ\Phi=f\circ (\Psi\circ\Phi)$ is holomorphic on $D$, because $f\in A_{D}(K)$. We conclude that $f\circ\Psi\in A_{D}(K_1)$. From the assumption we have $A_{D}(K_{1})=\overline{\mathcal O}(K_1)$. Thus, for any given $\varepsilon >0$, there is a function $q$ which is holomorphic on an open neighbourhood of $K_1$ such that
	\begin{equation*}
		\sup _{z\in K_{1}}|f\circ\Psi (z)-q(z)|<\varepsilon .
	\end{equation*}
	We consider the function $Q$, given by $Q(z,w)=q(z)$, which is holomorphic on a neighbourhood of $K$. Thus
	\begin{equation*}
		\sup _{(z,w)\in K} |f(z,w)-Q(z,w)|=\sup _{z\in K_{1}}|f\circ\Psi (z)-q(z)|<\varepsilon ,
	\end{equation*}
	and so $f\in \overline{\mathcal O}(K)$. This shows that $A_{D}(K)\subset\overline{\mathcal O}(K)$, which completes the proof. \end{proof}

	\begin{rem}
		\
		\
				 
		\begin{enumerate}
			\item Note that Example \ref{counterexample1} can be seen as a trivial graph where $K_{1}$ is the closed unit disc and $\omega\equiv 0$.
			\item Suppose that $K_1$ is a compact set in $\mathbb{C}$ with $K_{1}^{\circ}\neq\emptyset$. We note that if in Theorem  \ref{counterexample2} we replace the algebra $A_{D}$ by the algebra $A$ (both in the assumption and the conclusion), then the corresponding result is not valid; that is, $A(K)\neq \overline{\mathcal O}(K)$. To see this, assume $A(K_1)=\overline{\mathcal O}(K_1)$ where $K_{1}$ is any compact subset of $\mathbb C$ with non-empty interior and let $K=\{(z,w)\in \mathbb{C}^{2}: z\in K_{1}, w=\omega (z) \}$, for some function $w$ in $A(K_1)$. Consider the function $f(z,w)=|z|$. Then $f\in C(K)=A(K)$, because $K$ has empty interior in $\mathbb{C}^2$. Then, as in Example \ref{counterexample1} we can derive that $f$ is not in $\overline{\mathcal O}(K)$.

			\item Using  arguments similar to those in Theorem \ref{counterexample2} we can obtain the corresponding results if we replace (both in the assumption and the conclusion) the class $\overline{\mathcal O}$ with the classes $R$ or $P$.
			\item We note that if $K\subset\mathbb{C}^{d}$ is a closed polydisc or a Euclidean ball, then $A(K)=A_{D}(K)=\overline{\mathcal O}(K)=R(K)$ and the approximation can be realized by polynomials \cite{2013arXiv1304.5511M}.
			      			      
		\end{enumerate}
	\end{rem}

	\section{Applications}
	\label{Applications}
		In this section we provide certain approximation results for Cartesian products of an arbitrary (possibly infinite) indexed family of planar compact sets. For this purpose, in view of Proposition \ref{M}, we extend the definition of the algebra $A_{D}$ to arbitrary products.
		
	\begin{defn}
		\label{infiniteprod}
		Let $(K_i)_{i\in I}$ be an arbitrary family of compact subsets of $\mathbb C$ and $K=\prod_{i\in I} K_{i}$. A function $f:K\to \mathbb C$ is said to belong to the class $A_{D}(K)$ if the following conditions hold:
		\begin{itemize}
			\item[(a)] $f$ is continuous on $K$ endowed with the Cartesian topology,
			\item[(b)] For every $i_0\in I$, if we fix $w_i\in K_i$, $i\in I\setminus\{i_0\}$ then the function $K_{i_{0}}\ni z\to f\big((z_i(z)_{i\in I}\big)\in \mathbb C$ belongs to $A(K_{i_{0}})$, where $z_{i_0}(z)=z$ and $z_{i}(z)=w_i$ for all $i\in I\setminus\{i_0\}$.
		\end{itemize}
	\end{defn}
		
	\begin{thm}
		\label{thrm:mainA(Omega)}
		Let $\{K_i\}_{i\in I}$ be a family of planar compact sets satisfying $\overline{\mathcal O} (K_i)=A_D(K_i)$. Assume that for each $i\in I$, $L_i:=\{\alpha_{i,j},j\in J_i\}$ is a is a subset of $\mathbb C\cup\{\infty\}$ such that each connected component of $(\mathbb C\cup\{\infty\})\setminus K_i$ contains one of such elements. Then, if $K=\prod_{i\in I} K_i$, finite sums of finite products of rational functions of one variable $z_i$ with  poles in the set $L_i$ are uniformly dense in $A_{D}(K)$.
	\end{thm}
		
	\begin{proof}
		Let $f\in A_D(K)$ and $\varepsilon>0$. Then according to \cite[Lemma 4.3]{FalcoNestoridis2017}, there exists a finite set $F\subset I$ and an element $(\zeta_i)_{i\in I}\in K$ such that the function $g:K\to\mathbb C$ satisfies $\vert f(z)-g(z)\vert<\varepsilon/2$ for all $z=(z_i)_{i\in I}\in K$, where $g(z)=f\big((w_i(z))_{i\in I}\big)$, with $w(z_i)=z_i$ if $i\in F$ and $w(z_i)=\zeta_i$ if $i\in  I\setminus F$. Since $g$ can be seen as an element of $A_D(\prod_{i\in F}K_i)=R_{L}(\prod_{i\in F}K_i)$, according to Corollary \ref{answerquestion} we have that $g$ can be  $\varepsilon/2$ approximated by a function $h$ which is a finite sum of finite products of rational functions of one variable $z_{i}$ with poles in $L_i$. Then, we define $\tilde h(z)=h\big((z_i)_{i\in F}\big)$ for $z\in K$ and we have that $\vert f(z)-\tilde h(z)\vert<\varepsilon$ for all $z\in K$. This concludes the proof. 
	\end{proof}
		
	Given a product  $K=\prod_{i\in I}K_{i}$ of planar compact sets, we consider that a polynomial on $K$ is a finite sum of finite products of monomials depending only on one variable. As before we denote by $P(K)$ the set of uniform limits of polynomials on $K$. Moreover, a rational function on $K$ is a quotient $P/Q$ of two polynomials on $K$ with $Q\not\equiv 0$.
		
	\begin{rem}
		If in the previous theorem we assume that $\mathbb C\setminus K_i$ is connected for all $i\in I$, then the approximation can be realized by polynomials depending on a varying finite set of variables.
	\end{rem}
		
	\begin{defn}
		\label{infiniteR}
		Let $\{K_i\}_{i\in I}$ be a family of planar compact sets and $K$ a compact subset of $\prod_{i\in I}K_{i}$. A function $f:K\to \mathbb C$ is said to belong to the class $R(K)$ if it is a uniform limit on $K$ of rational functions depending on a finite number of variables with no singularities on $K$. Let $L_{i}$ be a subset of $(\mathbb C\cup\{\infty\})\setminus K_{i}$  for each $i\in I$. Then $R_{L}(K)$ denotes the set of uniform limits on $K$ of finite sums of finite products of rational functions depending on one variable $z_{i}$ with poles in $L_{i}$.
	\end{defn}

	\begin{cor}
		\label{cor:equivalence}
		Let $\{K_i\}_{i\in I}$ be a family of planar compact sets and  $K=\prod_{i\in I} K_i$. Let $L_{i}$, $i\in I$ be a subset of $(\mathbb C\cup\{\infty\})\setminus K_{i}$ containing at least one element from each component of $(\mathbb C\cup\{\infty\})\setminus K_{i}$. Then, the following assertions are equivalent:
		\begin{itemize}
			\item[(a)] $R(K)=A_D(K)$;
			\item[(b)] $R_{L}(K)=A_{D}(K)$;
			\item[(c)] $R(K_i)=A_{D}(K_i)$, for each $i\in I$.
		\end{itemize}
	\end{cor}
		
	\begin{proof}
		The implication $(c)\Rightarrow (b)$ is given by Theorem \ref{thrm:mainA(Omega)}. The implication $(b) \Rightarrow (a)$ is obvious since $R_L(K)\subset R(K) \subset A_{D}(K)$. Finally, the implication $(a)  \Rightarrow (c)$ is obtained by considering suitable functions in $A_{D}(K)$ depending only on one complex variable.
	\end{proof}

	\begin{rem}
		Under the above notation, the following assertions are equivalent:
		\begin{itemize}
			\item[(a)] $P(K)=A_D(K)$;
			\item[(b)] $\mathbb C\setminus K_i$ is connected for each $i\in I$.
		\end{itemize}
	\end{rem}

	\begin{cor}
		\label{cor:union}
		Let $\{K^{i}\}_{i\in I} $ be a family of planar compact sets with $K^{i}=\cup_{k\in I_{i}}K_{k,i}$,  where $\{K_{k,i}\}_{k\in I_{i}}$ is a family of pairwise disjoint planar sets and  $I_{i}$ is some index set. Denote by  $K=\prod_{i\in I} K^i$ and by $\{K_{j}\}_{j\in J}$ the family of all Cartesian products $\{\prod_{i\in I} K_{k,i}:k\in I_{i}\}$. Thus, we have that $\cup_{j\in J} K_j = K =\prod_{i\in I} K^i$. Assume that for each $i\in I$, $L_i$ is a set of points in $\mathbb C\cup\{\infty\}$ such that each connected component of $\mathbb C\setminus K^i$ contains one element of $L_{i}$. If $R(K)=A_{D}(K)$, then for every subset $H\subset J$ with $\cup_{j\in H} K_j$ compact and $\cup_{j\notin H} K_j$ compact, we have
		\begin{itemize}
			\item[(a)] $R(\cup_{j\in H}K_{j})=A_D(\cup_{j\in H}K_{j})$,
			\item[(b)]  $R_{L}(\cup_{j\in H}K_{j})=A_{D}(\cup_{j\in H}K_{j})$.
		\end{itemize}
	\end{cor}
	\begin{proof}
		Note that if $j_1\not=j_2$ then    $K_{j_1}\cap K_{j_2}= \emptyset$, since  the sets $\{K_{k,i}\}_{k\in I_{i}}$ are pairwise disjoint. For fixed $H\subset J$, we have that any function in $A_D(\cup_{j\in H}K_{j})$ can be extended to a function in $A_{D}(K)$ by using the fact that the sets $\{K_{j}\}_{j\in J}$ are pairwise disjoint and  $\cup_{j\in H} K_j$  and $\cup_{j\notin H} K_j$  are compact. For instance the extension can be defined as the zero function on the set $ \cup_{j\notin H}K_{j}$. Then the previous corollary gives us the desired result.
	\end{proof}
		
	Before we continue, let us show the following lemma which will be necessary for the proof of Theorem \ref{the:manyequivalences}.
		
	\begin{lem}
		\label{prodcomplexnumbers}
		Given  $\tilde\varepsilon>0$, if $z_{1},\ldots,z_{d}\in \mathbb C$ are such that $\vert z_{i}-1\vert\leq \tilde \varepsilon$ for $i=1,\ldots, d$, then $$\vert 1- \prod_{i=1}^{d}z_{i}\vert\leq (1+\tilde\varepsilon)^{d}-1.$$
	\end{lem}
		
	\begin{proof}
		We prove the result by induction on $d$. For $d=1$ the result follows immediately. Assume that the result holds for any set of $d-1$ complex numbers. Then, 
		\begin{align*}
			\vert 1- \prod_{i=1}^{d}z_{i}\vert & \leq \vert 1- z_{d}\vert +\vert z_{d}-\prod_{i=1}^{d}z_{j}\vert                      \\
			                                   & \leq \tilde\varepsilon +\vert z_{d}\vert\vert 1 - \prod_{i=1}^{d-1}z_{i}\vert        \\
			                                   & \leq \tilde\varepsilon +(1+\tilde\varepsilon)\Big((1+\tilde\varepsilon)^{d-1}-1\Big) \\
			                                   & =(1+\tilde\varepsilon)^{d}-1.                                                        
		\end{align*}
	\end{proof}

	For the purposes of our next result we need to extend Definition  \ref{infiniteprod} to more general subsets in $\mathbb{C}^I$, which are not necessarily products. If $K^j=\prod_{i\in I}K_i\subset \mathbb{C}^I$ is a product of planar compact sets $K_i$ and $K$ is a finite disjoint union of such products $K^j$, then for a function $f: K\rightarrow \mathbb{C}$ we say that $f\in A_{D}(K)$ if and only if $f|_{K^j}\in A_{D}(K^j)$ for all $j$. More generally, if $\mathbb C^{I}$ is endowed with the product topology and $K\subset \mathbb C^{I}$ is compact, we say that a function $f:K\to \mathbb C$ belongs to $A_{D}(K)$ if it is continuous on $K$, and for every $\Phi:D\to K\subset \mathbb C^{I}$ holomorphic on an open disc $D\subset \mathbb C$, the composition $f\circ \Phi: D\to \mathbb C$ is holomorphic on $D$. This definition in the case of products of planar compact sets implies Definition \ref{infiniteprod}. The converse is also true. To see this consider a function $f  : K\rightarrow \mathbb{C}$ satisfying the requirements of Definition \ref{infiniteprod}. Then, since $f$ is continuous on $K$ with respect to the product topology, it is the uniform limit on $K$ of a sequence of functions $f_{n}\in A_{D}(K)$ depending on a varying finite set of coordinates. Using Proposition \ref{M} we conclude that $f_{n}\circ \Phi$ is holomorphic on the disc $D$. Since $f_{n}\circ \Phi\to f\circ \Phi$ uniformly on $D$, as $n\to\infty$, we conclude that $f\circ \Phi$ is also holomorphic on $D$. Therefore, we see that in the case of products of planar compact sets the two definitions coincide, even in the infinite dimensional case.
		
	The next result is a generalization of Corollary \ref{cor:equivalence}. In this case we consider rational approximation on pairwise disjoint unions of products of planar sets  under some restrictions. Given a set $K$ and a subset $E$ of $K$, we denote by $\chi_{E}$ the function that is equal to one on $E$ and zero on $K\setminus E$.
	
	To help the reader comprehend the statement of Theorem \ref{the:manyequivalences} we first present a specific example which satisfies its geometrical assumptions.
	
	\begin{ex}
		Denote by $P\big((a_{1},a_{2}),(\rho_{1},\rho_{2})\big)$ the closed polydisc in $\mathbb C^{2}$ of center $(a_{1},a_{2})$ and polyradius $(\rho_{1}, \rho_{2})$, that is $\{(z_{1},z_{2})\in\mathbb C^{2}:\vert z_{1}-a_{1}\vert\leq \rho_{1}\text{ and }\vert z_{2}-a_{2}\vert\leq \rho_{2}\}$. Set 
		\begin{itemize}
		\item $K_{1}=P\big((35,55),(25,5)\big)$,
		\item $K_{2}=P\big((17,35),(7,5)\big)$,
		\item $K_{3}=P\big((35,15),(15,5)\big)$,
		\item $K_{4}=P\big((80,55),(10, 12)\big)$,
		\end{itemize}
		and 
		\[
			K=K_{1}\cup K_{2}\cup K_{3} \cup K_{4}.
		\]
		
		Then 
		\[
			K^{1}=\overline D(35,25)\cup \overline D(17,7)\cup \overline D(35,15)\cup \overline D(80,10)=\overline D(35,25)\cup \overline D(80,10),
		\] 
		and 
		\[
			K^{2}=\overline D(55,5)\cup \overline D(35,5)\cup \overline D(15,5)\cup \overline D(55,12)=\overline D(15,5)\cup \overline D(35,15)\cup \overline D(55,12).
		\]
		
		
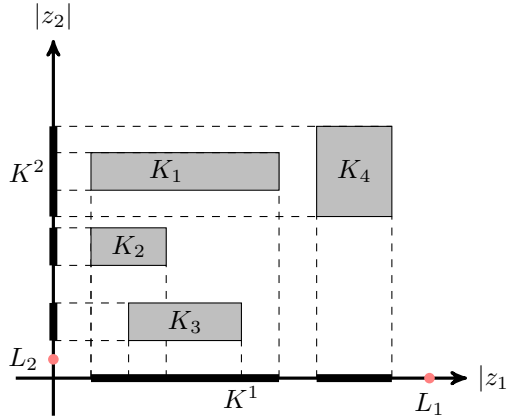
\begin{center}
			\begin{tikzpicture}[
					scale=5,
					axis/.style={very thick, ->, >=stealth'},
					important line/.style={thick},
					dashed line/.style={dashed, thin},
					pile/.style={thick, ->, >=stealth', shorten <=2pt, shorten
						>=2pt},
					every node/.style={color=black}
				]
				\draw[axis] (-0.1,0)  -- (1.1,0) node(xline)[right]
				{$\vert z_{1}\vert$};
				\draw[axis] (0,-0.1) -- (0,0.9) node(yline)[above] {$\vert z_{2}\vert$};
				\draw[draw=black, fill=gray!50] (0.1,0.5) rectangle (0.6,0.6);
				\node at (0.3,0.55) {$K_{1}$};
				\draw[dashed] (0.1,0) -- (0.1,0.5);
				\draw[dashed] (0.6,0) -- (0.6,0.5);
				
				\draw[dashed] (0,0.5) -- (0.1,0.5);
				\draw[dashed] (0,0.6) -- (0.6,0.6);
						
				\draw[draw=black, fill=gray!50] (0.1,0.3) rectangle (0.3,0.4);
				\node at (0.2,0.35) {$K_{2}$};
				\draw[dashed] (0.1,0) -- (0.1,0.3);
				\draw[dashed] (0.3,0) -- (0.3,0.3);
				
				\draw[dashed] (0,0.3) -- (0.1,0.3);
				\draw[dashed] (0,0.4) -- (0.1,0.4);
						
				\draw[draw=black, fill=gray!50] (0.2,0.1) rectangle (0.5,0.2);
				\node at (0.35,0.15) {$K_{3}$};
				\draw[dashed] (0.2,0) -- (0.2,0.1);
				\draw[dashed] (0.5,0) -- (0.5,0.1);
						
				\draw[dashed] (0,0.1) -- (0.2,0.1);
				\draw[dashed] (0,0.2) -- (0.2,0.2);

				\draw[draw=black, fill=gray!50] (0.7,0.43) rectangle (0.9,0.67);
				\node at (0.8,0.55) {$K_{4}$};
				\draw[dashed] (0.7,0) -- (0.7,0.43);
				\draw[dashed] (0.9,0) -- (0.9,0.43);
				
				\draw[dashed] (0,0.43) -- (0.7,0.43);
				\draw[dashed] (0,0.67) -- (0.9,0.67);
						
				\draw[line width=3pt] (0.1,0) -- (0.6,0);
				\draw[line width=3pt] (0.7,0) -- (0.9,0);
				\fill[fill=red!50] (0,0.05)  circle[radius=.4pt];
				\node at (0.5,-.05) {$K^{1}$};
				\node at (1,-.07) {$L_{1}$};

				\draw[line width=3pt] (0,0.1) -- (0,0.2);
				\draw[line width=3pt] (0,0.3) -- (0,0.4);
				\draw[line width=3pt] (0,0.43) -- (0,0.67);
				\fill[fill=red!50] (1,0)  circle[radius=.4pt];
				\node at (-0.07,0.55) {$K^{2}$};
				\node at (-.08,0.05) {$L_{2}$};

			\end{tikzpicture}
		\end{center}
		and we also consider $L_{1}=\{100\}$ and $L_{2}=\{5i\}$.
	\end{ex}
	
	\begin{thm}
		\label{the:manyequivalences}
		Let $\{K_{j}\}_{j=1}^{m}$ be a finite pairwise disjoint family of products of planar compacta, $K_{j}=\prod_{i\in I} K_{j,i}$, for a fixed set $I$. Let $K=\cup_{j=1}^{m}K_{j}$ and $K^{i}=\cup_{j=1}^{m}K_{j,i}$.
		Assume that, for all $i\in I $ and all $1\leq j_{1},j_{2},j_{3}\leq m$, if $K_{j_{1},i}\cap K_{j_{2},i}\ne \emptyset$ and $K_{j_{2},i}\cap K_{j_{3},i}\ne \emptyset$ then $K_{j_{1},i}\cap K_{j_{3},i}\ne \emptyset$. Assume also that, for each $i\in I$, there exists a set of points in $\mathbb C\cup\{\infty\}$
		, $L_{i}$, with $K^{i}\cap L_{i}=\emptyset$ for $i\in I$ and such that  each connected component of $\mathbb C\setminus K_{j,i}$ and each connected component of $\mathbb C\setminus K^{i}$ contains at least one such number.
					
		Then the following assertions are equivalent.
		\begin{itemize}
			\item[(a)] $R(K)=A_D(K)$;
			\item[(b)] $R_{L}(K)=A_D(K)$;
			\item[(c)] For each $j=1,\ldots,m$, $R(K_{j})=A_D(K_{j})$ and $\chi_{K_{j}}\in R(K)$;
			\item[(d)] For each $j=1,\ldots,m$, $R_{L}(K_{j})=A_D(K_{j})$ and $\chi_{K_{j}}\in R(K)$;
			\item[(e)] For each $j=1,\ldots,m$ and each $i\in I$, $R(K_{j,i})=A_{D}(K_{j,i})$ and $\chi_{K_{j}}\in R(K)$.
		\end{itemize}
	\end{thm}
		
	\begin{proof}
		We begin by showing that (a) implies (b) and we prove first the case where $I$ is a finite set of $d$ elements. Fix a function $h\in A_D(K)$ and a positive number $\varepsilon$. For each $j=1,\ldots,m$, since $R(K)=A_D(K)$ and the sets $\{K_{j}\}_{j=1}^{m}$ are pairwise disjoint, it is easy to see that $R( K_{j})=A_D(K_{j})$ for $j=1,\ldots,m$. By Corollary \ref{cor:equivalence} if we consider the function $h$ as a function defined on $K_{j}$,  there exists a function $g_j$ with 
		\begin{equation}
			\label{hminusgj}
			\Vert h - g_j\Vert_{K_{j}}\leq \varepsilon/2
		\end{equation} and $g_j$ can be written as a finite sum of finite products of rational functions of one variable $z_i$ with poles in the set $L_i$. Note that in particular the function $g_{j}$ is defined on $K$ since $K^{i}\cap L_{i}=\emptyset$. In fact, $g_{j}$ is bounded on $K$. Without loss of generality we can assume that $\Vert g_j\Vert_{K}\ne 0$.
		For each $i\in I$ and each $j=1,\ldots,m$, define the set 
		$$\displaystyle D_{j,i}=\bigcup_{K_{j,i}\cap K_{k,i}\ne \emptyset}\hspace*{-5mm}K_{k,i}.$$
		Then, $D_{j,i}$ and $K^{i}\setminus D_{j,i}$ are disjoint compact sets. Since $L_{i}$ contains at least one element in each connected component of $\mathbb C \setminus K^{i}$ and the function $\chi_{D_{j,i}}$ can be extended holomorphically on a neighborhood of $K^{i}$, as a consequence of  Runge's Theorem, we can find a rational function $r_{j,i}$ of one complex variable  with prescribed poles in the set $L_i$ such that 
		\begin{equation}
			\label{eq:epsilontilde}
			\Vert \chi_{D_{j,i}} - r_{j,i}\Vert_{ K^{i}}\leq \tilde\varepsilon,
		\end{equation}
		where  $\tilde \varepsilon$ is such that $\tilde \varepsilon(1+\tilde\varepsilon)^{d-1}\leq \dfrac{\varepsilon}{2m\Vert g_j\Vert_{ K}}$ and $(1+\tilde\varepsilon)^{d} -1 \leq \dfrac{\varepsilon}{2m\Vert g_j\Vert_{ K}}$.

		Consider the function $r_j(z_{1},\ldots,z_{d})=\prod_{i=1}^d r_{j,i}(z_{i})$. We have that  $K_{\alpha}\cap K_{j}=\emptyset$ for each  $\alpha\ne j$ because the product sets $K_{1},\ldots,K_{m}$ are pairwise disjoint. In particular, since they are product sets we have that there exists  $i_{0}$ depending on $\alpha$ and $j$ so that $K_{\alpha,i_{0}}\cap K_{j,i_{0}}=\emptyset$. Since for any $1\leq j_{1},j_{2},j_{3}\leq m$ if $K_{j_{1},i}\cap K_{j_{2},i}\ne \emptyset$ and $K_{j_{2},i}\cap K_{j_{3},i}\ne \emptyset$, then, $K_{j_{1},i}\cap K_{j_{3},i}\ne \emptyset$, we have that $K_{\alpha,i_{0}}\cap D_{j,i_{0}}=\emptyset$.
		Thus,  $\chi_{D_{j,i_{0}}}=0$ on the set $K_{\alpha,i_{0}}$ and by \eqref{eq:epsilontilde}   we obtain that $\Vert r_{j, i_{0}}\Vert_{K_{\alpha,i_{0}}}\leq \tilde \varepsilon$. Also, by \eqref{eq:epsilontilde}  again,  for any $i\ne i_{0}$ and for every $\alpha=1,\ldots,m$ we have that $\Vert r_{j,i}\Vert_{K_{\alpha,i}}\leq 1+\tilde\varepsilon$. Hence, for any $\alpha\ne j$, 
		\begin{equation*}
			\Vert \chi_{K_{j}}-r_{j}\Vert_{K_{\alpha}}=\Vert r_{j}\Vert_{K_{\alpha}}=\Vert \prod_{i=1}^{d}r_{j,i}\Vert_{K_{\alpha}}\leq\Vert r_{j,1}\Vert_{K_{\alpha,1}}\cdots \Vert r_{j,d}\Vert_{K_{\alpha,d}}\leq \tilde\varepsilon(1+\tilde\varepsilon)^{d-1}\leq \frac{\varepsilon}{2m\Vert g_j\Vert_{ K}}.
		\end{equation*}
		Also, for $ i=1,\dots , d$ and $z_i\in K_{j,i}\subset D_{j,i}$, using \eqref{eq:epsilontilde} we get
		$$
		\vert 1-r_{j,i}(z_i)\vert=\vert \chi_{D_{j,i}}(z_i)-r_{j,i}(z_i)\vert\leq \Vert \chi_{D_{j,i}}-r_{j,i}\Vert_{K^{i}}\leq\tilde \varepsilon. 
		$$
		Hence, by Lemma \ref{prodcomplexnumbers} applied to the complex numbers $r_{j,1}(z_1),	\ldots,r_{j,d}(z_d)$ we have that 
		\[
			\vert 1- \prod_{i=1}^{d}r_{j,i}(z_i)\vert\leq (1+\tilde\varepsilon)^{d}-1
		\]
		for all $z_1\in K_{j,1}, \ldots, z_d\in K_{j,d}$. Therefore,
		\begin{equation*}
			\Vert \chi_{K_{j}}-r_{j}\Vert_{K_{j}}=\Vert 1- r_{j}\Vert_{K_{j}}=\Vert 1- \prod_{i=1}^{d}r_{j,i}\Vert_{K_{j}}=\sup_{\substack{ z_i\in K_{j,i}\\ i=1,\ldots, d}}\vert 1- \prod_{i=1}^{d}r_{j,i}(z_i)\vert\\\leq (1+\tilde\varepsilon)^{d}-1\leq \frac{\varepsilon}{2m\Vert g_j\Vert_{ K}}.
		\end{equation*}	
		Hence we have that
		\begin{equation}
			\label{eq:approxrationprod}
			\Vert \chi_{K_{j}} - r_j\Vert_{K}\leq \frac{\varepsilon}{2m\Vert g_j\Vert_{ K}}.
		\end{equation}
				
		Now, consider the function 
		\begin{equation*}
			g=\sum_{j=1}^m r_jg_j.
		\end{equation*}
		Clearly $g$ is a  finite sum of finite products of rational functions of one variable $z_i$ with poles in the set $L_i$ and 
		\begin{align*}
			\Vert h - g \Vert_{ K} & = \Vert h - \sum_{j=1}^m r_jg_j \Vert_{ K}\leq \varepsilon, 
		\end{align*}
		where the last inequality comes from the fact that for each $s\in\{1,\ldots,m\}$
		\begin{align}
			\label{inequalities}
			\Vert h - \sum_{j=1}^m r_jg_j\Vert_{K_s}
			  & \leq  \Vert h - r_sg_s\Vert_{K_s} + \Vert \sum_{\substack{ j=1                                                   \\j\ne s}}^m r_jg_j\Vert_{K_s}\nonumber\\
			  & \leq  \Vert h - g_s \Vert_{K_s} + \Vert g_s - r_sg_s\Vert_{K_s} +  \sum_{\substack{ j=1                          \\j\ne s}}^m \Vert r_jg_j\Vert_{K_{s}}\\
			  & \leq \frac{\varepsilon}{2} + \frac{\varepsilon}{2m\Vert g_s\Vert_{K}}\Vert g_s\Vert_{K_s} + \sum_{\substack{ j=1 \\j\ne s}}^m \frac{\varepsilon}{2m\Vert g_j\Vert_{K}}\Vert g_j\Vert_{K_s}
			\leq \varepsilon \hspace{.8cm}(\text{by }\eqref{hminusgj}\text{ and } \eqref{eq:approxrationprod}).\nonumber
		\end{align}
		Hence the proof for the case with $I$ finite is complete. 
					

		Consider now the case where $I$ is not (necessarily) finite. Since $K$ is a finite union of product sets, similarly to \cite[Lemma 4.3]{FalcoNestoridis2017}, we have that  the set of functions depending on a finite number of variables is dense in $A_D(K)$. Therefore, for a fixed function $h\in A_D(K)$ and a positive number $\varepsilon$ we can find a function $\tilde h\in A_D(K)$ that depends only on a finite number of variables $z_1, \ldots, z_d$ with $\Vert h-\tilde h\Vert\leq \varepsilon/2$. By the previous case, if we consider the function $\tilde h$ as a function defined on a set $\tilde K =\cup_{j=1}^{m}\prod_{i=1}^d K_{j,i}$ we have that there exists a function $g$ defined on $\tilde K$ with $g$ being a finite sum of finite products of rational functions of one variable $z_i$ with  poles in the set $L_i$ and $\Vert \tilde h - g\Vert_{\tilde K}\leq  \varepsilon /2$. By considering $g$ as a function defined on $K$ we have that $\Vert h-g\Vert_{ K} \leq \varepsilon$. Hence (a) implies (b). The reverse implication is obvious by considering the common denominator of the finite sum of finite products of rational functions depending on one variable. Hence (a) and (b) are equivalent.
					
		Statements (c), (d) and (e) are equivalent by Corollary \ref{cor:equivalence}. It is easy to see that  (a) implies (c) because the sets $\{K_{j}\}_{j=1}^{m}$ are pairwise disjoint. Indeed, since $A_{D}(K)=R(K)$ implies that $A_{D}(K_{j})=R(K_{j})$ for all $j=1,\ldots,m$ and since $\chi_{K_{j}}\in A_{D}(K)=R(K)$ the conclusion holds. 
					
		To finish the proof we show that (e) implies (a). Let us fix $h\in A_D(K)$ and a positive number $\varepsilon$. For each $j=1,\ldots,m$, we consider $h_{j}$ to be the restriction of $h$ to $K_{j}$. Clearly $h_{j}\in A_{D}(K_{j})$. By Corollary \ref{cor:equivalence} we can find a function $g_{j}$ with $\Vert h_{j}-g_{j}\Vert_{K_{j}}\leq \varepsilon/2$ and $g_{j}$ is a finite sum of finite products of rational functions of one variable $z_i$ with  poles in the set $L_i$. Since $K^{i}\cap L_{i}=\emptyset$ for all $i\in I$, the function $g_{j}$ is bounded on $K$. Using that $\chi_{K_{j}}\in R(K)$ we can find a rational function $r_{j}$ with singularities  off $K$ such that $\Vert g_{j}-r_{j}g_{j}\Vert_{K_{j}}\leq \varepsilon/(2m)$ and $\Vert r_{j}g_{j}\Vert_{K_{s}}\leq \varepsilon/(2m)$ for $j\ne s$. The function 
		$$
		q=\sum_{s=1}^{m}r_{s}g_{s}
		$$ is a rational function with singularities off $K$. The same argument as in equation \eqref{inequalities} yields $\Vert h-q\Vert_{K}\leq \varepsilon$. This completes the proof.
	\end{proof}
		
	\begin{rem}
				
		Senechkin showed that in the previous theorem the family $\{K_{j}\}_{j=1}^{m}$ needs to be finite. See \cite{MR0382722}  for the details. 				
				
	\end{rem}
		
	To conclude we show that if $K$ is rationally convex then the hypothesis in Theorem \ref{the:manyequivalences} (c), (d), (e)  that  $\chi_{K_{j}}\in R(K)$ is satisfied.
	Denote by $K^{\wedge_{r}}$ the rationally convex hull of a compact set $K\subset\mathbb C^{d}$, that is 
	the set of points $z\in\mathbb C^{d}$ such that 
	\[
		\vert g(z)\vert\leq \max_{x\in K}\vert g(x)\vert 
	\]
	for every rational function $g$ which is holomorphic on a neighbourhood of  $K$. We recall that $K$ is rationally convex if $K=K^{\wedge_{r}}.$ If $K$ is rationally convex then every holomorphic function on (a neighborhood of) $K$ can be uniformly approximated by rational functions $p/q$ where $q$ is zero-free on $K$, by the Oka-Weil approximation theorem, see  \cite[Page 95]{MR0410387}  for the details.  
			  
	It can be seen that the rationally convex hull of a product set is the product of the rationally convex hulls of the factors. It follows that if $K_i\subset \mathbb C^{n_i}, i=1,\ldots,d,$ then
	\[
		\Big(\prod_{i=1}^dK_i\Big)^{\wedge_{r}} = \prod_{i=1}^{d}K_i^{\wedge_{r}}.
	\]
	In particular, a finite product of planar compact sets is always rationally convex. 		
		
	The condition of being rationally convex is not necessary to obtain rational approximation. The sphere is not rationally convex, while every holomorphic function on the sphere can be uniformly approximated by polynomials. However, as we shall show now, rational convexity is also a sufficient condition to ensure that Theorem \ref{the:manyequivalences} holds.	
		
	\begin{cor}
		Let $I=\{1,\ldots,N\}$ be a finite set, $L_i, i\in I,$  subsets of $\mathbb C\cup\{\infty\})$ and $\{K_{j}\}_{j=1}^{m}$ a finite pairwise disjoint family of products of planar compacta satisfying the hypotheses of  Theorem \ref{the:manyequivalences}. 
		Then, if $K_{j,i}$ is a set of rational approximation for each $j=1,\ldots,m$ and each $i\in I$ and $K$ is rationally convex, we have that finite sums of finite products of rational functions of one variable $z_i$ with  poles in the set $L_{i}$ are uniformly dense in $A_{D}(K)$. In particular rational functions are dense in $A_D(K)$.
	\end{cor}
		
	\begin{proof}
		We shall prove that under the hypothesis of the corollary, condition (e) in Theorem \ref{the:manyequivalences} is satisfied. This would finish the proof since conditions (e) and (b) are equivalent in view of Theorem  \ref{the:manyequivalences}.
				
		$K_{j,i}$ is a set of rational approximation for each $j=1,\ldots,m$ and each $i\in I$ which means that $R(K_{j,i})=A_D(K_{j,i})$.  Also, the sets $\{K_{j}\}_{j=1}^{\infty}$ are pairwise disjoint and compact, hence the function $\chi_{K_{j}}$ can be naturally extended to be holomorphic on a neighborhood of the set $K$. Since $K$ is a rationally convex set, by the rational Oka-Weil theorem we can approximate $\chi_{K_{j}}$ by rational functions having no singularities on $K$. Therefore $\chi_{K_{j}}\in R(K)$. Thus, the requirements of condition (e) in Theorem  \ref{the:manyequivalences} are fulfilled and the result holds.
	\end{proof}
		
	\begin{rem}
		It is not in general true that if we have finitely many products with $A_{D}(K_{j})=P(K_{j})$ then $A_{D}(\cup_{j=1}^{m} K_{j})=P(\cup_{j=1}^{m}K_{j})$. Kallin \cite{Kallin} showed that there exist three congruent, pairwise disjoint, closed polydiscs $P_{1}$, $P_{2}$ and $P_{3}$ in $\mathbb C^{3}$ such that $P_{1}\cup P_{2}\cup P_{3}$ is not polynomially convex. Kallin's proof actually used polydiscs parallel to the coordinate axes.
	\end{rem}

	\section{Appendix: All inclusions are in general strict}
	\label{appendix}

	We recall that
	\[
		P(K)\ \stackon{\subset}{\scriptscriptstyle{(1)}}\  R(K) \ \stackon{\subset}{\scriptscriptstyle{(2)}}\   \overline{\mathcal O}(K)\ \stackon{\subset}{\scriptscriptstyle{(3)}}\   A_{D}(K)\ \stackon{\subset}{\scriptscriptstyle{(4)}}\   A(K) \ \stackon{\subset}{\scriptscriptstyle{(5)}}\   C(K).
	\]
	All these inclusions are in general strict. We shall focus on dimension  $d=3$. For $(1)$ we consider  $\mathbb T\times \{0\}\times\{0\}$ where $\mathbb T$ is the unit circle in $\mathbb C$. For $(2)$ we shall give an example in $\mathbb C^{3}$ (see Example \ref{zeronex}). For $(3)$ we may consider  $S\times \{0\}\times\{0\}$ where $S$ is the Swiss cheese in $\mathbb C$. Finally, for $(4)$ we consider $\overline{\mathbb D}\times\{0\}\times \{0\}$ where $\overline{\mathbb D}$ is the closed unit disc in $\mathbb C$ and argue as in Example \ref{counterexample1}. For $(5)$ we take the closed unit ball in $\mathbb C^{3}$.

	It only remains to give an example for $(2)$. For this we shall use  maximal ideal spaces, see \cite{MR0410387}. If $\mathcal F$ is a function algebra on a compact subset $K\subset \mathbb C^n,$ denote by $\mathcal M(\mathcal F)$ 
	the maximal ideal space of $\mathcal F$. Recall that, if $\mathcal F=C(K),$ then $\mathcal  M(\mathcal F) = K$; if $\mathcal F = P(K),$ then $\mathcal M(\mathcal F) $ is the polynomial hull $\widehat K$, and; if $\mathcal F = R(K),$ then $\mathcal M(\mathcal F)$ is the rational hull $K^{\wedge_{r}}.$

	Hence 
	\[
		\widehat K =\mathcal M(P(K))\supset   K^{\wedge_{r}}= \mathcal M(R(K)) \supset \mathcal M(\overline{\mathcal O}(K))\supset\mathcal M(A_{D}(K))\supset\mathcal M(A(K))\supset \mathcal M(C(K))=K.
	\]
	An immediate consequence is that if $C(K)=P(K),$ then we have equalities for all of above inclusions and hence $K$ is polynomially convex. Moreover, if $K$ is not rationally convex, then $C(K)\not = R(K).$

	\begin{ex}
		\label{zeronex}
		We consider $K$ to be a totally real compact manifold in $\mathbb C^3$ that is not rationally convex. For the existence of such a $K$ see \cite{AH}. In \cite[Corollary 3.4]{HW} it is shown that if $K$ is a closed totally real $C^1$-submanifold of an open set in $\mathbb C^3,$ then $\overline{\mathcal O}(K) = C(K).$  Since $K$ is not rationally convex we have that  $R(K)$  is a proper subset of $C(K)$. Therefore, 
		\[
			\overline{\mathcal O}(K)=C(K)\ne R(K).
		\]
	\end{ex}
	We note that the previous example leads to an example of an open set $V\subset \mathbb C^{3}$ such that 		\[
	R(V)\ne \mathcal O(V).
\]	
	 
\begin{rem}
	By considering the union of suitable translations of the above sets we can obtain a compact set $K$ in $\mathbb C^{3}$ such that inclusions $(1)$ to $(5)$ are strict simultaneously.
\end{rem}

	

\end{document}